\documentclass[10pt,reqno]{amsart} 

\usepackage{amsthm,amsfonts,amscd,amsmath}
\usepackage[hidelinks]{hyperref} 
\usepackage{url}
\usepackage{xcolor}
\hypersetup{
    colorlinks,
    linkcolor={red!50!black},
    citecolor={blue!50!black},
    urlcolor={blue!80!black}
}

\usepackage[normalem]{ulem}
\usepackage{graphicx} 
\usepackage{datetime} 
\usepackage{arydshln} 
\usepackage{cancel}
\usepackage{caption,subcaption}
\numberwithin{figure}{section}
\numberwithin{table}{section}

\newcommand{\gkpSI}[2]{\ensuremath{\genfrac{\lbrack}{\rbrack}{0pt}{}{#1}{#2}}} 
\newcommand{\gkpSII}[2]{\ensuremath{\genfrac{\lbrace}{\rbrace}{0pt}{}{#1}{#2}}}

\newcommand{\cf}{\textit{cf.\ }} 
\newcommand{\Iverson}[1]{\ensuremath{\left[#1\right]_{\delta}}}

\DeclareMathOperator{\fg}{fg}
\DeclareMathOperator{\Div}{div}

\title[Factorization Theorems for Hadamard Products and Derivatives]{
       Factorization Theorems for Hadamard Products and 
       Higher-Order Derivatives of Lambert Series Generating Functions 
} 
\author[M. D. Schmidt]{Maxie D. Schmidt \\ 
         Georgia Institute of Technology \\ 
         School of Mathematics \\ 
         117 Skiles Building \\ 
         686 Cherry Street NW \\ 
         Atlanta, GA 30332 \\ 
         USA \\ 
         \href{mailto:maxieds@gmail.com}{maxieds@gmail.com} \\ 
         \href{mailto:mschmidt34@gatech.edu}{mschmidt34@gatech.edu} 
} 

\email{\href{mailto:maxieds@gmail.com}{maxieds@gmail.com} \\ 
       \href{mailto:mschmidt34@gatech.edu}{mschmidt34@gatech.edu} 
}
\address{Georgia Institute of Technology \\ 
         School of Mathematics \\ 
         117 Skiles Building \\ 
         686 Cherry Street NW \\ 
         Atlanta, GA 30332 \\ 
         USA
} 

\date{\today} 


\allowdisplaybreaks 

\theoremstyle{plain} 
\newtheorem{theorem}{Theorem}

\newtheorem{prop}[theorem]{Proposition}
\newtheorem{lemma}[theorem]{Lemma}
\newtheorem{cor}[theorem]{Corollary}
\numberwithin{theorem}{section}

\theoremstyle{definition} 
\newtheorem{example}[theorem]{Example}
\newtheorem{remark}[theorem]{Remark}
\newtheorem{definition}[theorem]{Definition}

\begin{document} 

\begin{abstract} 
We first summarize joint work on several preliminary canonical Lambert series factorization 
theorems. Within this article we establish new analogs to these 
original factorization theorems which characterize two specific primary cases of the 
expansions of Lambert series generating functions: factorizations for 
Hadamard products of Lambert series and for higher-order derivatives of Lambert series. 
The series coefficients corresponding to these two generating function cases are 
important enough to require the special due attention we give to their expansions 
within the article, and moreover, are significant in that they connect the 
characteristic expansions of Lambert series over special multiplicative functions to the 
explicitly additive nature of the theory of partitions. 
Applications of our new results provide new exotic sums involving 
multiplicative functions, new summation-based interpretations of the coefficients of the 
integer-order $j^{th}$ derivatives of Lambert series generating functions, 
several new series for the Riemann zeta function, and an exact identity for the 
number of distinct primes dividing $n$. 
\end{abstract}

\maketitle

\section{Introduction} 
\label{Section_Intro}

\subsection{Lambert series factorization theorems} 

In the references we have proved several variants and generalized expansions of 
\emph{Lambert series factorization theorems} of the form 
\cite{MERCA-SCHMIDT3,MERCA-SCHMIDT2,MERCA-LSFACTTHM,SCHMIDT-LSFACTTHM} 
\begin{align*} 
\sum_{n \geq 1} \frac{a_n q^n}{1-q^n} & = 
     \frac{1}{(q; q)_{\infty}} \sum_{n \geq 1} \sum_{k=1}^n s_{n,k} a_k \cdot q^n, 
\end{align*} 
and of the form 
\begin{align*}
\sum_{n \geq 1} \frac{\bar{a}_n q^n}{1-q^n} & = 
     \frac{1}{(q; q)_{\infty}} \sum_{n \geq 1} \sum_{k=1}^n \widetilde{s}_{n,k} \widetilde{a}_k \cdot q^n, 
\end{align*} 
for $\widehat{a}_n$ and $\bar{a}_n$ 
depending on an arbitrary arithmetic function $a_n$ and where the lower-triangular sequence 
$s_{n,k} := [q^n] (q; q)_{\infty} q^k / (1-q^k)$, which we typically require to be independent of the $a_n$, 
is the difference of two partition 
functions counting the number of $k$'s in their respective odd (even) distinct partitions. 

In the concluding remarks to \cite{MERCA-SCHMIDT3} we gave several specific examples of 
other constructions of related Lambert series factorization theorems. 
In the reference, we also proved a few new properties of the factorizations of 
Lambert series generating functions over the convolution of two arithmetic functions, 
$f \ast g$ expanded by 
\begin{align*} 
\sum_{n \geq 1} \frac{(f \ast g)(n) q^n}{1-q^n} & = \frac{1}{(q; q)_{\infty}} 
     \sum_{n \geq 1} \sum_{k=1}^n 
     \widetilde{s}_{n,k}(g) f(k) \cdot q^n, 
\end{align*} 
which we cite in this article. 
The Lambert series factorizations of these two forms and their variations, two to three 
of which we consider in this article, also imply matrix factorizations of these 
expansions which are dictated by the corresponding typically invertible matrix of 
$s_{n,k}$ or $\widetilde{s}_{n,k}(g)$. The matrix-based interpretation of these factorization 
theorems is perhaps the most intuitive way to explore how these expansions ``\emph{factor}'' into 
distinct matrices applied to vectors of special sequences. 

\subsection{Significance of our results} 

Our new results are rare and important because they provide a mechanism that 
effectively translates the divisor sums of the coefficients of a Lambert series generating function 
into ordinary sums which similarly generate the same prescribed arithmetic function, say $a_n$. 
Moreover, these factorization theorems connect the special functions in multiplicative number theory 
which are typically tied to a particular Lambert series expansion with the additive theory of 
partitions and partition functions in unusual and unexpected new ways. 
We are one of the first authors to examine such relations between the additive and multiplicative 
in detail (see also \cite{MERCA-SCHMIDT-AMM,MERCA-SCHMIDT3,MERCA-SCHMIDT2}). 
We note that we are not the first to consider the derivatives of Lambert series generating functions 
\cite{SCHMIDT-BDDDIVSUMS}, though 
our perspective on the connections afforded by these factorizations is distinctly new. 

\subsection{Focus within this article} 

Within this article we explore the expansions of factorization theorems for two primary 
additional special case 
variants which are distinctive and important enough in their applications to require special attention here: 
Hadamard products of two Lambert series generating functions and the higher-order integer derivatives 
of Lambert series generating functions. Section \ref{Section_HP} proves several new properties of the 
first case, where the results proved in Section \ref{Section_LSDerivs} consider the second case in detail. 
The significance of these two particular factorizations is that they have a broad range of applications to 
expanding special and classical arithmetic functions from multiplicative number theory. 
In Section \ref{Section_OtherFactThms} we tie up loose ends by offering two other related 
variants of the Lambert series factorizations. Namely, we prove factorization theorems for 
generating function convolutions and provide a purely matrix-based proof of a new formula 
for the coefficients enumerated by a Lambert series generating function. 

\subsubsection*{New results and characterizations} 

The Hadamard product generating function cases lead to several forms of 
new so-termed ``\emph{exotic}'' sums for classical special functions as illustrated in the explicit 
corollaries given in Example \ref{example_HPThm_consequences} of the next section. 
For example, if we form the Hadamard product of generating functions of the two 
Lambert series over Euler's totient function, $\phi(n)$, we obtain the following more 
exotic-looking sum for our multiplicative function of interest: 
\[
\phi(n) = \sum_{k=1}^n \sum_{d|n} \frac{p(d-k)}{d} \mu(n/d)\left[k^2 + 
     \sum_{b=\pm 1} \sum_{j=1}^{\left\lfloor \frac{\sqrt{24k-23}-b}{6} \right\rfloor} (-1)^j 
     \left(k-\frac{j(3j+b)}{2}\right)^2\right]. 
\]
The importance of 
obtaining new formulas and identities for the derivatives of a series whose coefficients we are 
interested in studying should be obvious, though our factorization results also provide new alternate 
characterizations of these derivatives apart from typical ODEs which we can form involving the 
derivatives of any sequence ordinary generating functions. 
Our results for the factorization theorems for derivatives of Lambert series generating functions 
in Section \ref{Section_LSDerivs} provides two particular 
factorizations which characterize these expansions. 

\section{Factorization Theorems for Hadamard Products} 
\label{Section_HP}

\subsection{Hadamard products of generating functions} 

The \emph{Hadamard product} of two ordinary generating functions $F(q)$ and $G(q)$, respectively 
enumerating the sequences of $\{f_n\}_{n \geq 0}$ and $\{g_n\}_{n \geq 0}$ is defined by 
\begin{align*}
(F \circ G)(q) & := \sum_{n \geq 0} f_n g_n \cdot q^n,\ \text{ for $|q| < \sigma_F \sigma_G$, } 
\end{align*} 
where $\sigma_F$ and $\sigma_G$ denote the radii of convergence of each respective generating function. 
Analytically, we have an integral formula and corresponding coefficient extraction formula 
for the Hadamard product of two generating functions when $F(q)$ is expandable in a fractional series 
respectively given by \cite[\S 1.12(V); Ex.\ 1.30, p.\ 85]{ADVCOMB} \cite[\S 6.3]{ECV2} 
\begin{align*} 
(F \circ G)(q^2) & = \frac{1}{2\pi} \int_0^{\pi} F(q e^{\imath t}) G(q e^{-\imath t}) dt \\ 
(F \circ G)(q) & = [x^0] F\left(\frac{q}{x}\right) G(x). 
\end{align*} 
In the context of the factorization theorems we consider in this article and in the references, 
we consider the Hadamard products of two Lambert series generating functions for special 
arithmetic functions $f_n$ and $g_n$ which we define coefficient-wise to enumerate the product of the 
divisor sums over each sequence corresponding to the coefficients of the individual Lambert series 
over the two functions. 
This subtlety is discussed shortly in 
Definition \ref{def_HP_for_LSGFs}. 
As we prove below, it turns out that we can formulate analogous factorization theorems for the cases 
of these Hadamard products as well. 
The next definition makes the expansion of the particular Hadamard product functions we consider in 
this section more precise. 

\begin{definition}[Hadamard Products for Lambert Series Generating Functions] 
\label{def_HP_for_LSGFs}
For any fixed arithmetic functions $f$ and $g$, we define the Hadamard product of the 
two Lambert series over $f$ and $g$ to be the auxiliary Lambert series generating function 
over the composite function $a_{\fg}(n)$ whose coefficients are given by 
\[
\sum_{d|n} a_{\fg}(d) = [q^n] \sum_{m \geq 1} \frac{a_{\fg}(m) q^m}{1-q^m} := 
     \underset{:=\fg(n)}{\underbrace{\left(\sum_{d|n} f_d\right)\left(\sum_{d|n} g_d\right)}}, 
\] 
so that by M\"obius inversion we have that 
\[
a_{\fg}(n) = \sum_{d|n} \fg(d) \mu(n/d) = (\fg \ast \mu)(n). 
\] 
We note that this definition of our Hadamard product generating functions is already 
slightly different than the 
one given above in that we define the Hadamard product of two Lambert series generating functions 
by the expansion of a third composite Lambert series which corresponds to the 
particular expansion of the factorization in \eqref{eqn_HPFactGenExp} below. 
\end{definition} 

\subsection{Main results and applications} 

The next theorems in this section define the key matrix sequences, 
$s_{n,k}(f)$ and $s_{n,k}^{(-1)}(f)$, in terms of the 
next factorization of the Lambert series over $a_{\fg}(n)$ from the definition above in the form of 
\begin{align} 
\label{eqn_HPFactGenExp} 
\sum_{n \geq 1} \frac{a_{\fg}(n) q^n}{1-q^n} & = \frac{1}{(q; q)_{\infty}} \sum_{n \geq 1} 
     \sum_{k=1}^n s_{n,k}(f) g_k \cdot q^n, 
\end{align} 
where $s_{n,k}(f)$ is independent of the function $g$, 
which is equivalent to defining the factorization expansion by the inverse matrix sequences as 
\begin{align} 
\label{eqn_HPFactGenExp_v2} 
g_n & = \sum_{k=1}^n s_{n,k}^{(-1)} \cdot 
     [q^k] \left((q; q)_{\infty} \times \sum_{n \geq 1} \frac{a_{\fg}(n) q^n}{1-q^n}\right). 
\end{align} 
We also define the following function to expand divisor sums over arithmetic functions as 
ordinary sums for any integers $1 \leq k \leq n$: 
\[
T_{\Div}(n, k) := 
     \begin{cases} 
     1, & \text{ if $k|n$; } \\ 
     0, & \text{ otherwise. }
     \end{cases} 
\] 

\begin{theorem} 
For all integers $1 \leq k \leq n$, we have the following definition of the factorization matrix 
sequence defining the expansion on the right-hand-side of \eqref{eqn_HPFactGenExp} where 
we adopt the notation $\widetilde{f}(n) := \sum_{d|n} f_d$: 
\begin{align*} 
 & s_{n,k}(f) = T_{\Div}(n, k) \widetilde{f}(n) \\ 
 & \phantom{=\quad\ } + \sum_{b=\pm 1} 
     \sum_{j=1}^{\left\lfloor \frac{\sqrt{24(n-k)+1}-b}{6} \right\rfloor} 
     (-1)^j T_{\Div}\left(n-\frac{j(3j+b)}{2}, k\right) \cdot 
     \widetilde{f}\left(n-\frac{j(3j+b)}{2}\right).
\end{align*} 
\end{theorem} 
\begin{proof} 
By the factorization in \eqref{eqn_HPFactGenExp} and the definition of $a_{\fg}(n)$ given above, we have 
that for $\widetilde{f}(n) = \sum_{d|n} f_d$ 
\begin{align*} 
s_{n,k}(f) & = [g_k] \left(\sum_{d|n} f_d\right) \times \sum_{d=1}^n g_d T_{\Div}(n, d) \\ 
     & = 
     [q^n] (q; q)_{\infty} \times \sum_{n \geq 1} T_{\Div}(n, k) \widetilde{f}(n) q^n, 
\end{align*} 
which equals the stated expansion of the sequence by Euler's pentagonal number theorem which provides that 
\begin{align*} 
(q; q)_{\infty} & = 1 + \sum_{j \geq 1} (-1)^j \left(q^{j(3j-1)/2} + q^{j(3j+1)/2}\right). 
     \qedhere 
\end{align*} 
\end{proof} 

\begin{theorem}[Inverse Sequences] 
\label{theorem_HP_InvSeqs} 
For all integers $1 \leq k \leq n$, we have the next definition of the inverse factorization matrix 
sequence which equivalently defines the expansion on the right-hand-side of \eqref{eqn_HPFactGenExp}. 
\[
s_{n,k}^{(-1)}(f) = \sum_{d|n} \frac{p(d-k)}{\widetilde{f}(d)} \mu(n/d), 
\] 
\end{theorem} 
\begin{proof} 
We expand the right-hand-side of the factorization in \eqref{eqn_HPFactGenExp} 
for the sequence $g_n := s_{n,r}^{(-1)}(f)$, i.e., the exact inverse sequence, 
for some fixed $r \geq 1$ as follows\footnote{ 
     \underline{\emph{Notation}}: 
     \emph{Iverson's convention} compactly specifies 
     boolean-valued conditions and is equivalent to the 
     \emph{Kronecker delta function}, $\delta_{i,j}$, as 
     $\Iverson{n = k} \equiv \delta_{n,k}$. 
     Similarly, $\Iverson{\mathtt{cond = True}} \equiv 
                 \delta_{\mathtt{cond}, \mathtt{True}}$ 
     in the remainder of the article. 
}: 
\begin{align*} 
\widetilde{f}(n) \cdot \sum_{d|n} s_{d,r}^{(-1)}(f) & = 
     \sum_{j=0}^n \sum_{k=1}^j s_{j,k} s_{k,r}^{(-1)} \cdot p(n-j) \\ 
     & = 
     \sum_{j=0}^n \Iverson{j = r} p(n-j) = p(n-r). 
\end{align*} 
Then the last equation implies that 
$$\sum_{d|n} s_{d,r}^{(-1)}(f) = \frac{p(n-r)}{\widetilde{f}(n)}, $$ 
which by M\"obius inversion implies our stated result. 
\end{proof} 

\begin{example}[Applications of the Theorem] 
\label{example_HPThm_consequences}
For the arithmetic function pairs 
$$(f, g) := (n^t, n^s), (\phi(n), \Lambda(n)), (n^t, J_t(n)), $$ 
respectively, and some constants $s,t \in \mathbb{C}$ where $\sigma_{\alpha}(n)$ denotes the 
generalized sum-of-divisors function, $\Lambda(n)$ is von Mangoldt's function, 
$\phi(n)$ is Euler's totient function, and $J_t(n)$ is the Jordan totient function, 
we employ the equivalent expansions of the factorization result in \eqref{eqn_HPFactGenExp_v2} 
to formulate the following ``\emph{exotic}'' sums as consequences of the theorems above: 
\begin{align}
\label{eqn_nsPow_Sigmast_HPIdent_v1} 
n^s & = \sum_{k=1}^n \sum_{d|n} \frac{p(d-k)}{\sigma_t(d)} \mu(n/d)\Biggl[\sigma_t(k)\sigma_s(k) \\ 
\notag 
     & \phantom{=\sum\ } + 
     \sum_{b=\pm 1} \sum_{j=1}^{\left\lfloor \frac{\sqrt{24k+1}-b}{6} \right\rfloor} (-1)^j 
     \sigma_t\left(k-\frac{j(3j+b)}{2}   \right)\sigma_s\left(k-\frac{j(3j+b)}{2}\right)\Biggr] \\ 
\notag 
\Lambda(n) & = \sum_{k=1}^n \sum_{d|n} \frac{p(d-k)}{d} \mu(n/d)\Biggl[k\log(k) \\ 
\notag 
     & \phantom{=\sum\ } + 
     \sum_{b=\pm 1} \sum_{j=1}^{\left\lfloor \frac{\sqrt{24k-23}-b}{6} \right\rfloor} (-1)^j 
     \left(k-\frac{j(3j+b)}{2}\right)\log\left(k-\frac{j(3j+b)}{2}\right)\Biggr] \\ 
\notag 
J_t(n) & = \sum_{k=1}^n \sum_{d|n} \frac{p(d-k)}{d^t} \mu(n/d)\left[k^{2t} + 
     \sum_{b=\pm 1} \sum_{j=1}^{\left\lfloor \frac{\sqrt{24k-23}-b}{6} \right\rfloor} (-1)^j 
     \left(k-\frac{j(3j+b)}{2}\right)^{2t}\right].  
\end{align}
By forming a second divisor sum over the divisors of $n$ on both sides of the first equation above, 
the first more exotic-looking sum for the sum-of-divisors functions leads to an expression for 
$\sigma_s(n)$ as a sum over the paired product functions, $\sigma_t(n) \cdot \sigma_s(n)$. 
We do not know of another such identity relating the generalized sum-of-divisors functions existing in the 
literature or in the references which we cite in this article. 
     However, the following relations between the multiplicative 
     generalized sum-of-divisors functions and 
     special additive partition functions are known where $p_k(n)$ denotes the number of 
     partitions of $n$ into at most $k$ parts and $\operatorname{pp}(n)$ denotes the number of 
     plane, or planar, partitions of $n$ \cite[A000219]{OEIS} \cite[\S 26.9, \S 26.12]{NISTHB}: 
     \begin{align*} 
     n \cdot p(n) & = \sum_{k=0}^{n-1} p(k) \sigma_1(n) \\ 
     n \cdot p_k(n) & = \sum_{t=1}^n p_k(n-t) \sum_{\substack{j|t \\ j \leq k}} j,\ k \geq 1 \\ 
     n \cdot \operatorname{pp}(n) & = \sum_{j=1}^n \operatorname{pp}(n-j) \sigma_2(j). 
     \end{align*}  
\end{example} 

\begin{cor}[New Series for the Riemann Zeta Function] 
\label{cor_SumOfDivFns_ZetaFnIdents}
For fixed $s,t \in \mathbb{C}$ such that $\Re(s) > 1$ 
we have the following infinite sum representations of the Riemann zeta function 
where we denote the sequence of interleaved pentagonal numbers, $\omega(\pm n)$, by 
$G_j = \frac{1}{2} \left\lceil \frac{j}{2} \right\rceil \left\lceil \frac{3j+1}{2} \right\rceil$ for 
$j \geq 0$ \cite[A001318]{OEIS}: 
\begin{align*} 
\zeta(s) & = \sum_{n \geq 1} \sum_{k=1}^n \sum_{d|n} \frac{p(d-k)}{\sigma_t(d)} \mu(n/d) \times 
     \sum_{j: G_j < k} (-1)^{\lceil j/2 \rceil} \frac{\sigma_t(k-G_j) \sigma_s(k-G_j)}{(k-G_j)^s} \\ 
\zeta(s) & = \sum_{n \geq 1} \sum_{k=1}^n \sum_{d|n} \frac{d^t \cdot p(d-k)}{\sigma_t(d)} \mu(n/d) \times 
     \sum_{j: G_j < k} (-1)^{\lceil j/2 \rceil} \frac{\sigma_t(k-G_j) \sigma_s(k-G_j)}{(k-G_j)^{s+t}}.  
\end{align*} 
\end{cor} 
\begin{proof} 
These two identities follow as special cases of the theorem in the form of 
\eqref{eqn_nsPow_Sigmast_HPIdent_v1} above where we note the identity for the generalized 
sum-of-divisors functions which provides that $\sigma_{-\alpha}(n) = \sigma_{\alpha}(n) / n^{\alpha}$ 
for all $\alpha \in \mathbb{C}$. 
We note that the pentagonal number theorem employed in the inner sums depending on $j$ 
is equivalent to the expansion 
\[
(q; q)_{\infty} = \sum_{j \geq 0} (-1)^{\lceil j/2 \rceil} q^{G_j}. 
\] 
The convergence of these infinite series is guaranteed by our hypothesis that $\Re(s) > 1$. 
\end{proof} 

We compare the results in the previous theorem to the known Dirichlet generating functions 
for the sum-of-divisors functions which are expanded by 
\cite[Thm.\ 291]{HARDYANDWRIGHT} \cite[\S 27.4]{NISTHB} 
\begin{align*} 
\zeta(s) \zeta(s-\alpha) & = \sum_{n \geq 1} \frac{\sigma_{\alpha}(n)}{n^s},\ \Re(s) > 1, \alpha+1 \\ 
\frac{\zeta(s) \zeta(s-\alpha) \zeta(s-\beta) \zeta(s-\alpha-\beta)}{\zeta(2s-\alpha-\beta)} & = 
     \sum_{n \geq 1} \frac{\sigma_{\alpha}(n) \sigma_{\beta}(n)}{n^s},\ \Re(s) > 1, \alpha + 1, \beta + 1. 
\end{align*} 
In particular, we note that while the series $\sum_n \sigma_{s}(n) / n^s$, and similarly 
for the second series, are divergent, our sums given in 
Corollary \ref{cor_SumOfDivFns_ZetaFnIdents} do indeed converge for $\Re(s) > 1$. 

\section{Factorization Theorems for Derivatives of Lambert Series} 
\label{Section_LSDerivs}

\subsection{Derivatives of Lambert series generating functions} 

\begin{figure}[ht!]

\begin{minipage}{\linewidth} 
\begin{center} 
\small
\begin{equation*} 
\boxed{ 
\begin{array}{cccccccccccc}
 1 & 0 & 0 & 0 & 0 & 0 & 0 & 0 & 0 & 0 & 0 & 0 \\
 1 & 2 & 0 & 0 & 0 & 0 & 0 & 0 & 0 & 0 & 0 & 0 \\
 0 & -2 & 3 & 0 & 0 & 0 & 0 & 0 & 0 & 0 & 0 & 0 \\
 -1 & 2 & -3 & 4 & 0 & 0 & 0 & 0 & 0 & 0 & 0 & 0 \\
 -2 & -4 & -3 & -4 & 5 & 0 & 0 & 0 & 0 & 0 & 0 & 0 \\
 -2 & 2 & 6 & -4 & -5 & 6 & 0 & 0 & 0 & 0 & 0 & 0 \\
 -2 & -4 & -6 & 0 & -5 & -6 & 7 & 0 & 0 & 0 & 0 & 0 \\
 -1 & 2 & -3 & 8 & 0 & -6 & -7 & 8 & 0 & 0 & 0 & 0 \\
 0 & -2 & 9 & -4 & 0 & 0 & -7 & -8 & 9 & 0 & 0 & 0 \\
 1 & 2 & -6 & -8 & 15 & 0 & 0 & -8 & -9 & 10 & 0 & 0 \\
 2 & 0 & -3 & 4 & -10 & 6 & 0 & 0 & -9 & -10 & 11 & 0 \\
 3 & 2 & 12 & 12 & -5 & 12 & 7 & 0 & 0 & -10 & -11 & 12 \\
\end{array}
}
\end{equation*}
\end{center} 
\subcaption*{\rm{(i)} $s_{1,n,k}$} 
\end{minipage} 

\begin{minipage}{\linewidth} 
\begin{center} 
\small
\begin{equation*} 
\boxed{ 
\begin{array}{cccccccccccc}
 1 & 0 & 0 & 0 & 0 & 0 & 0 & 0 & 0 & 0 & 0 & 0 \\
 -\frac{1}{2} & \frac{1}{2} & 0 & 0 & 0 & 0 & 0 & 0 & 0 & 0 & 0 & 0 \\
 -\frac{1}{3} & \frac{1}{3} & \frac{1}{3} & 0 & 0 & 0 & 0 & 0 & 0 & 0 & 0 & 0 \\
 \frac{1}{4} & 0 & \frac{1}{4} & \frac{1}{4} & 0 & 0 & 0 & 0 & 0 & 0 & 0 & 0 \\
 0 & \frac{3}{5} & \frac{2}{5} & \frac{1}{5} & \frac{1}{5} & 0 & 0 & 0 & 0 & 0 & 0 & 0 \\
 1 & 0 & \frac{1}{6} & \frac{1}{3} & \frac{1}{6} & \frac{1}{6} & 0 & 0 & 0 & 0 & 0 & 0 \\
 \frac{4}{7} & 1 & \frac{5}{7} & \frac{3}{7} & \frac{2}{7} & \frac{1}{7} & \frac{1}{7} & 0 & 0 & 0 & 0 & 0 \\
 \frac{9}{8} & \frac{7}{8} & \frac{5}{8} & \frac{3}{8} & \frac{3}{8} & \frac{1}{4} & \frac{1}{8} & \frac{1}{8} & 0 & 0 & 0 & 0 \\
 \frac{16}{9} & \frac{4}{3} & \frac{8}{9} & \frac{7}{9} & \frac{5}{9} & \frac{1}{3} & \frac{2}{9} & \frac{1}{9} & \frac{1}{9} & 0 & 0 & 0 \\
 \frac{5}{2} & \frac{11}{10} & \frac{11}{10} & \frac{9}{10} & \frac{1}{2} & \frac{1}{2} & \frac{3}{10} & \frac{1}{5} & \frac{1}{10} & \frac{1}{10} & 0 & 0 \\
 \frac{31}{11} & \frac{30}{11} & 2 & \frac{15}{11} & 1 & \frac{7}{11} & \frac{5}{11} & \frac{3}{11} & \frac{2}{11} & \frac{1}{11} & \frac{1}{11} & 0 \\
 \frac{13}{4} & \frac{8}{3} & \frac{7}{4} & \frac{5}{4} & \frac{13}{12} & \frac{3}{4} & \frac{7}{12} & \frac{5}{12} & \frac{1}{4} & \frac{1}{6} & \frac{1}{12} & \frac{1}{12} \\
\end{array}
} 
\end{equation*} 
\end{center} 
\subcaption*{\rm{(ii)} $s_{1,n,k}^{(-1)}$} 
\end{minipage}

\caption{Factorization matrix sequences for the first-order derivatives 
         of an arbitrary Lambert series generating function when $1 \leq n,k \leq 12$} 
\label{table_FirstOrderDerivs_MatrixSeqExamples} 

\end{figure} 

We seek analogous factorization theorems for the higher-order $t^{th}$ derivatives for all 
integers $t \geq 1$ of an arbitrary Lambert series in the form of 
\begin{align} 
\label{eqn_LSDerivs_FirstFactorizationAttempt} 
q^t \cdot D_t\left[\sum_{n \geq t} \frac{a_n q^n}{1-q^n}\right] & = 
     \frac{1}{(q; q)_{\infty}} \sum_{n \geq t} \sum_{k=t}^n s_{t,n,k} a_k \cdot q^n. 
\end{align} 
We note that we consider these sums over $n \geq t$ to produce an invertible matrix of the 
factorization sequences $s_{t,n,k}$. 
At first computation, the corresponding matrix sequences for these higher-order derivatives 
do not suggest any immediate intuitions to exact formulas for $s_{t,n,k}$ and $s_{t,n,k}^{(-1)}$ 
as in the factorizations expanded above. The listings given in 
Figure \ref{table_FirstOrderDerivs_MatrixSeqExamples} provide the first several rows of these 
sequences in the first-order derivative case where $t := 1$ for comparison with our intuition. 
However, may still prove using the method invoked in the proof of 
Theorem \ref{theorem_HP_InvSeqs} to show that in the first-order case we have that\footnote{ 
     More generally, if we expand the next mixed series of $j^{th}$ derivatives and 
     initial terms annihilated by the differential operator as 
     \[
     q^j \cdot D_j\left[\sum_{n \geq 1} \frac{a_n q^n}{1-q^n}\right] + 
     \sum_{i=1}^{j-1} (a \ast 1)(i) q^i = 
     \frac{1}{(q; q)_{\infty}} \sum_{n \geq t} \sum_{k=t}^n s_{j,n,k} a_k \cdot q^n, 
     \] 
     for integers $j \geq 2$, we can easily prove that 
     $s_{j,n,k} = s_{n,k} = [q^n] q^k / (1-q^k) (q; q)_{\infty}$ for $1 \leq n < j$ and 
     consequently that 
     \[ 
     s_{j,n,k}^{(-1)} = \sum_{d|n} \frac{p(d-k)}{\binom{d}{j} \cdot j! + 
     \delta_{d,1} + \delta_{d,2} + \cdots + \delta_{d,j-1}} \mu(n/d), 
     \] 
     using the proof method in Theorem \ref{theorem_HP_InvSeqs}. 
}
\begin{align*} 
s_{1,n,k}^{(-1)} & = \sum_{d|n} \frac{p(d-k)}{d} \mu(n/d), 
\end{align*} 
which, in light of the construction of Corollary \ref{cor_SumOfDivFns_ZetaFnIdents} and its 
notation, leads to the following further 
convergent series infinite expansion of the Riemann zeta function for $\Re(s) > 1$: 
\begin{align*} 
\zeta(s) & = \sum_{n \geq 1} \sum_{k=1}^n \sum_{d|n} \frac{p(d-k)}{d} \mu(n/d) \times 
     \sum_{j: G_j < k} (-1)^{\lceil j/2 \rceil} \frac{\sigma_s(k-G_j)}{(k-G_j)^{s-1}}. 
\end{align*} 
If we dig deeper into the expansions of the derivatives of arbitrary Lambert series, 
we can prove other more natural 
exact formulas for both matrix sequences which are for the most part actually 
just restatements of consequences of known factorization theorems already proved in the 
references \cite[\S 4]{MERCA-SCHMIDT3}. We consider the factorizations of these higher-order 
derivative cases again as a separate topic in this article 
due to the significance of the interpretations 
and the breadth of applications which we can give by explicitly defining the exact 
factorization expansions in these cases below. 

\subsection{Main results} 

We will next require the statements of the next two results proved in \cite{SCHMIDT-BDDDIVSUMS} to 
state and prove the main results in this section. We refer the reader to the proofs of these 
two lemmas given in the reference. 

\begin{lemma}[Modified Lambert Series Coefficients] 
For any fixed arithmetic function $a_n$ and integers $m,t \geq 1$, $n, k \geq 0$, we have the 
following identity for the series coefficients of modified Lambert series generating function 
expansions: 
\begin{align} 
\label{eqn_BinCvlDivFn_as_CoeffsOfLSeries_exp} 
[q^n] \sum_{i \geq t} \frac{a_i q^{mi}}{(1-q^i)^{k+1}} & = 
     \sum_{\substack{d|n \\ t \leq d \leq \left\lfloor \frac{n}{m} \right\rfloor}} 
     \binom{\frac{n}{d}-m+k}{k} a_d. 
\end{align} 
\end{lemma} 

\begin{lemma}[Higher-Order Derivatives of Lambert Series] 
\label{lemma_sthDerivsOfLambertSeries} 
For any fixed non-zero $q \in \mathbb{C}$, $i \in \mathbb{Z}^{+}$, and 
prescribed integer $s \geq 0$, we have the following two results\footnote{ 
     \underline{\emph{Notation}}: 
     The bracket notation of $\gkpSI{n}{k} \equiv (-1)^{n-k} s(n, k)$ 
     denotes the unsigned triangle of Stirling numbers of the first kind and 
     $\gkpSII{n}{k} \equiv S(n, k)$ denotes the triangle of Stirling numbers of the 
     second kind.
}: 
\begin{align} 
\tag{i} 
q^s D^{(s)}\left[\frac{q^i}{1-q^i}\right] & = 
     \sum_{m=0}^s \sum_{k=0}^m \gkpSI{s}{m} \gkpSII{m}{k} 
     \frac{(-1)^{s-k} k! \cdot i^m}{(1-q^i)^{k+1}} \\ 
\tag{ii} 
q^s D^{(s)}\left[\frac{q^i}{1-q^i}\right] & = 
     \sum_{r=0}^s \left( 
     \sum_{m=0}^s \sum_{k=0}^m \gkpSI{s}{m} \gkpSII{m}{k} \binom{s-k}{r} 
     \frac{(-1)^{s-k-r} k! \cdot i^m}{(1-q^i)^{k+1}} 
     \right) q^{(r+1)i}. 
\end{align} 
\end{lemma} 

\begin{prop}
\label{prop_Atn_LambertSeriesGFs} 
For integers $n,t \geq 1$, let the function $A_t(n)$ be defined as follows: 
\begin{align*} 
A_t(n) & := \sum_{\substack{0 \leq k \leq m \leq t \\ 0 \leq r \leq t}} \sum_{d|n} 
     \gkpSI{t}{m} \gkpSII{m}{k} \binom{t-k}{r} \binom{\frac{n}{d}-1-r+k}{k} 
     (-1)^{t-k-r} k! \times \\ 
     & \phantom{:= \sum_{\substack{0 \leq k \leq m \leq t \\ 0 \leq r \leq t}} \sum_{d|n}\ } \times 
     d^m a_d \cdot \Iverson{t \leq d \leq \left\lfloor \frac{n}{r+1} \right\rfloor}. 
\end{align*} 
Then we have the next two Lambert series expansions for the function $A_t(n)$ given by 
\[
A_t(n) = [q^n] q^t \cdot D_t\left[\sum_{n \geq t} \frac{a_n q^n}{1-q^n}\right] = 
     [q^n] \sum_{n \geq 1} \frac{(A_t \ast \mu)(n) q^n}{1-q^n}. 
\] 
\end{prop} 
\begin{proof} 
The first equation follows from \eqref{eqn_BinCvlDivFn_as_CoeffsOfLSeries_exp} 
applied to Lemma \ref{lemma_sthDerivsOfLambertSeries} when $s := t$. 
To prove the second form of a Lambert series generating function over some sequence $c_n$ 
enumerating $A_t(n)$, we require that 
$$\sum_{d|n} c_d = A_t(n), $$ which is true if and only if 
$$c_d = \sum_{d|n} A_t(d) \mu(n/d) = (A_t \ast \mu)(n), $$ by M\"obius inversion. 
\end{proof} 

\begin{theorem}[Higher-Order Derivatives of Lambert Series Generating Functions] 
\label{theorem_LSDerivs}
Let the notation for the function $A_t(n)$ be defined as in 
Proposition \ref{prop_Atn_LambertSeriesGFs}. Then we have the next formulas for 
$A_t(n)$ given by 
\begin{align*} 
A_t(n) & = [q^n] \frac{1}{(q; q)_{\infty}} \sum_{n \geq 1} \sum_{k=1}^n 
     \widetilde{s}_{n,k}(\mu) A_t(k) \cdot q^n \\ 
A_t(n) & = \sum_{k=1}^n \widetilde{s}_{n,k}^{(-1)}(\mu) \left[ 
     A_t(k) + \sum_{p = \pm 1} \sum_{j=1}^{\left\lfloor \frac{\sqrt{24k-23}-p}{6} \right\rfloor} 
     (-1)^j A_t\left(k-\frac{j(3j+p)}{2}\right)\right], 
\end{align*} 
where $$\widetilde{s}_{n,k}(\mu) = \sum_{j=1}^n s_{n,kj} \cdot \mu(j), $$ for 
$s_{n,k} := s_o(n,k) - s_e(n,k)$  when the 
sequences $s_o(n, k)$ and $s_e(n, k)$ respectively denote the number of $k$'s in all 
partitions of $n$ into an odd (even) number of distinct parts, and where a 
$n$-fold convolution formula involving $\mu$ for the inverse matrix sequence 
$\widetilde{s}_{n,k}^{(-1)}(\mu)$ is proved explicitly 
in the references \cite[\S 4]{MERCA-SCHMIDT3}. 
Moreover, for all $n \geq 1$ we have that the full Lambert series $t^{th}$ derivative 
formula is given by 
\begin{align*} 
\frac{n!}{(n-t)!} \cdot \sum_{d|n} a_d & = \left(\sum_{i=1}^{t-1} 
     \sum_{k=1}^{\left\lfloor \frac{n}{i} \right\rfloor} 
     \widetilde{s}_{n,ik}^{(-1)}(\mu) \cdot \frac{(ik)! a_i}{(ik-t)!}\right) + 
     A_t(n). 
\end{align*} 
\end{theorem} 
\begin{proof} 
The first result follows from the factorizations of the Lambert series over the convolution of 
two arithmetic functions proved in the reference \cite[\S 4]{MERCA-SCHMIDT3} where our 
Lambert series expansion in question is provided by
Proposition \ref{prop_Atn_LambertSeriesGFs} above. Similarly, the second result is a 
consequence of the first whose explicit expansions, i.e., for the inverse sequence are 
again proved in the reference. The last equation in the theorem follows from the 
proposition and adding back in the subtracted Lambert series terms when the summation for the 
series considered for $A_t(n)$ starts from $n \geq t$ instead of from one. 
The multiples of $ik$ in the last formula reflect that the coefficients of $q^i / (1-q^i)$ and its 
$q^t$-scaled derivatives are always zero unless the coefficient index is a multiple of $i$. 
\end{proof} 

\subsection{Another related factorization} 

\begin{remark}[Another Factorization] 
The first factorization expansion we considered in \eqref{eqn_LSDerivs_FirstFactorizationAttempt} 
of this section is obtained by applying Lemma \ref{lemma_ARelatedFactResult} in the case where 
\begin{align*} 
b_{n,i} & = [a_i] (A_t \ast \mu)(n) \\ 
     & = 
     \sum_{\substack{0 \leq k \leq m \leq t \\ 0 \leq r \leq t}} \sum_{d|n} 
     \gkpSI{t}{m} \gkpSII{m}{k} \binom{t-k}{r} \binom{\frac{d}{i}-1-r+k}{k} 
     (-1)^{t-k-r} k! \times \\ 
     & \phantom{:= \sum_{\substack{0 \leq k \leq m \leq t \\ 0 \leq r \leq t}} \sum_{d|n}\ } \times 
     i^m T_{\Div}(d, i) \mu(n/d) \cdot 
     \Iverson{t \leq i \leq \left\lfloor \frac{d}{r+1} \right\rfloor}. 
\end{align*} 
In this case, we can obtain a similar expansion of the middle identity in 
Theorem \ref{theorem_LSDerivs} in the form of 
\begin{align*} 
a_n & = \sum_{k=1}^n s_{t,n,k}^{(-1)}(b) \left[ 
     A_t(k) + \sum_{p = \pm 1} \sum_{j=1}^{\left\lfloor \frac{\sqrt{24k-23}-p}{6} \right\rfloor} 
     (-1)^j A_t\left(k-\frac{j(3j+p)}{2}\right)\right]. 
\end{align*} 
\end{remark} 

\begin{lemma}[A Related Factorization Result] 
\label{lemma_ARelatedFactResult}
If we expand the Lambert series factorization 
\[
\sum_{n \geq 1} \frac{\sum_{j=1}^n b_{n,j} a_j \cdot q^n}{1-q^n} = 
     \frac{1}{(q; q)_{\infty}} \sum_{n \geq 1} \sum_{k=1}^n s_{n,k}(b) a_k \cdot q^n, 
\]
then we have the formula 
\[
s_{n,k}(b) = \sum_{j=1}^n s_{n,j} \cdot b_{j,k}, 
\] 
where $s_{n,k} := [q^n] (q; q)_{\infty} q^k / (1-q^k) = s_o(n, k) - s_e(n, k)$ when the 
sequences $s_o(n, k)$ and $s_e(n, k)$ respectively denote the number of $k$'s in all 
partitions of $n$ into an odd (even) number of distinct parts. 
\end{lemma} 
\begin{proof} 
If we take the nested coefficients first of $a_k$ and then of $b_{j,k}$ for some $j,k \geq 1$ 
on both sides of the factorization cited above, we obtain that 
\[
\frac{q^j}{1-q^j} \cdot (q; q)_{\infty} = [b_{j,k}] \sum_{n \geq 1} s_{n,k}(b) \cdot q^n. 
\] 
Then if we take the coefficients of $q^n$ on each side of the previous equation we arrive at the 
identity 
\[
s_{n,j} = [b_{j,k}] s_{n,k}(b), 
\] 
for $j = 1,2,\ldots,n$. Finally, we multiply through both sides of the last equation by $b_{j,k}$ and then 
sum over all $j$ to conclude that the stated formula for $s_{n,k}(b)$ in the lemma is correct. 
Equivalently, since both sequences of $b_{n,k}$ and $s_{n,k}(b)$ are lower triangular, we could have 
deduced this identity by truncating the partial sums up to $n$ and employing a matrix argument to 
justify the formula above. 
\end{proof} 

\section{Expansions of other special factorization theorems} 
\label{Section_OtherFactThms}

\subsection{A factorization theorem for convolutions of Lambert series} 

In what follows we adopt the next notation for the Lambert series over a prescribed 
arithmetic function $h$ defined by 
\[
H_L(q) := \sum_{n \geq 1} \frac{h(n) q^n}{1-q^n}. 
\]
We seek a factorization theorem for the convolution of two Lambert series generating functions, 
$F_L(q)$ and $G_L(q)$, in the following form: 
\begin{align} 
\label{eqn_LSCvlFactThmExp} 
\frac{1}{q} \cdot F_L(q) G_L(q) & = \frac{1}{(q; q)_{\infty}} 
     \sum_{n \geq 1} \sum_{k=1}^n s_{n,k}(g) f_k \cdot q^n. 
\end{align} 

\begin{theorem}[Factorization Theorem for Convolutions] 
For the Lambert series factorization defined in \eqref{eqn_LSCvlFactThmExp}, we have the 
following exact expansions of the two matrix sequences characterizing the 
factorization where the difference of partition functions 
$s_{n,k} := [q^n] (q; q)_{\infty} q^k / (1-q^k)$: 
\begin{align*} 
s_{n,k}(g) & = \sum_{j=1}^{n+1} s_{j,k} \left(\sum_{d|n+1-j} g(d)\right) \\ 
s_{n,k}^{(-1)}(g) & = \sum_{d|n} [q^d]\left(\frac{q^{k+1}}{(q; q)_{\infty} G_L(q)}\right) \mu(n/d). 
\end{align*} 
\end{theorem} 
\begin{proof} 
We prove our second result along the same lines as the proof of 
Theorem \ref{theorem_HP_InvSeqs} above. 
Namely, we choose $F_L(q)$ to denote the Lambert series over 
$s_{n,k}^{(-1)}(g)$ for some fixed $k \geq 1$ and expand the right-hand-side of 
\eqref{eqn_LSCvlFactThmExp} as 
\begin{align*} 
[q^n] \frac{q^k}{(q; q)_{\infty} \cdot G_L(q) / q} & = \sum_{d|n} s_{d,k}^{(-1)}(g), 
\end{align*} 
which by M\"oebius inversion implies our stated result. 
Then to obtain a formula for the sequence $s_{n,k}(g)$ from the identity expanding the 
inverse sequence, we observe a property of the product of any two inverse matrices which is that 
\begin{align*} 
\Iverson{n=k} & = \sum_{j=1}^n s_{n,j}^{(-1)}(g) s_{j,k}(g). 
\end{align*} 
We then perform a divisor sum over $n$ in the previous equation to obtain that 
\begin{align*} 
\Iverson{k|n} & = \sum_{d|n} s_{d,k}^{(-1)}(g) \\ 
     & = 
     \sum_{j=1}^n [q^{n-j-1}] \frac{1}{(q; q)_{\infty} \cdot G_L(q)} \times s_{j,k}(g), 
\end{align*} 
which by another generating function argument implies that 
\begin{align*} 
s_{n,k}(g) & = [q^n] \left(\sum_{n \geq 1} \Iverson{k|n} q^n\right) (q; q)_{\infty} \cdot 
     \frac{G_L(q)}{q} \\ 
     & = 
     \frac{q^k}{1-q^k} (q; q)_{\infty} \cdot \frac{G_L(q)}{q} \\ 
     & = 
     \sum_{j=1}^{n+1} s_{j,k} \left(\sum_{d|n+1-j} g(d)\right), 
\end{align*} 
as claimed. 
\end{proof} 

We notice that the factorization in \eqref{eqn_LSCvlFactThmExp} together with the 
theorem imply that we have the two expansions of the following form: 
\begin{align*} 
[q^n] F_L(q) & = \sum_{k=1}^n \sum_{d|n} [q^d] \frac{q^{k+1}}{(q; q)_{\infty} G_L(q)} \times 
     \mu(n/d) \cdot [q^k] (q; q)_{\infty} F_L(q) G_L(q) \\ 
[q^n] G_L(q) & = \sum_{k=1}^n \sum_{d|n} [q^d] \frac{q^{k+1}}{(q; q)_{\infty} G_L(q)} \times 
     \mu(n/d) \cdot [q^k] (q; q)_{\infty} G_L(q)^2. 
\end{align*} 
The special case where $F_L(q) := G_L(q)$ in the last expansion provides a curious new 
relation between any Lambert series generating function $G_L(q)$, its reciprocal, and its square. 
This observation can be iterated to obtain even further multiple sum identities involving 
powers of $G_L(q)$. 

\subsection{A matrix-based proof of a factorization for sequences generated by Lambert series} 

As a last application of special cases of the Lambert series factorization theorems we have 
extended in this article, we consider another method of matrix-based proof which provides 
new formulas for the sequences generated by a Lambert series over $a_n$: $b(n) = (a \ast 1)(n)$. 
This approach is unique because unlike the factorization theorem variants we have 
explored so far which provide new identities and expansions for the sequence $a_n$ itself, the 
result in Theorem \ref{theorem_matrix_methods_omegan_genstmt} 
provides useful new inverse sequence expansions exclusively for the terms $b(n)$ whose 
ordinary generating function is the Lambert series generating function at hand 
\cite[\cf Thm.\ 1.4]{SCHMIDT-LSFACTTHM}. 

The first factorization theorem expanded in the introduction implicitly provides a 
matrix-multiplication-based representation of the coefficients $b(n)$ stated in terms of the 
matrix, $(T_{\Div}(i, j))_{n \times n} \equiv (T_{\Div})_n$, in the explicit forms of 
\begin{align*} 
(T_{\Div})_n \begin{bmatrix} a_1 \\ a_2 \\ \vdots \\ a_n \end{bmatrix} = \sum_{d|n} a_d 
     \quad\text{ and }\quad 
(T_{\Div})_n^{-1} \begin{bmatrix} a_1 \\ a_2 \\ \vdots \\ a_n \end{bmatrix} & = 
     \sum_{d|n} a_d \cdot \mu(n/d), 
\end{align*} 
where the corresponding inverse operation above is M\"obius inversion. 
For example, when $n := 6$ these matrices are given by 
\begin{align*} 
(T_{\Div})_6 = 
\begin{bmatrix} 
 1 & 0 & 0 & 0 & 0 & 0 \\
 1 & 1 & 0 & 0 & 0 & 0 \\
 1 & 0 & 1 & 0 & 0 & 0 \\
 1 & 1 & 0 & 1 & 0 & 0 \\
 1 & 0 & 0 & 0 & 1 & 0 \\
 1 & 1 & 1 & 0 & 0 & 1 \\
\end{bmatrix} 
     \quad\text{ and }\quad 
(T_{\Div})_6^{-1} = 
\begin{bmatrix} 
 1 & 0 & 0 & 0 & 0 & 0 \\
 -1 & 1 & 0 & 0 & 0 & 0 \\
 -1 & 0 & 1 & 0 & 0 & 0 \\
 0 & -1 & 0 & 1 & 0 & 0 \\
 -1 & 0 & 0 & 0 & 1 & 0 \\
 1 & -1 & -1 & 0 & 0 & 1 \\
\end{bmatrix} 
\end{align*} 
Operations with our new definitions of the matrices above allow us to prove the next result. 

\begin{theorem} 
\label{theorem_matrix_methods_omegan_genstmt} 
For all $n \geq 1$ and a fixed arithmetic function $a_n$ we have the identity 
\[
b(n) = \sum_{k=1}^n \sum_{j=1}^k s_{n,k} C_{k,j} a_j, 
\] 
where the inner matrix entries are given by \cite[A000041]{OEIS} 
\[
C_{n,k} = \sum_{d|n} \sum_{i=1}^d p(d-ik) \mu(n/d). 
\] 
\end{theorem} 
\begin{proof} 
We first note that the theorem is equivalent to showing that we have a desired expansion of the form 
\[
\tag{i} 
\sum_{k=1}^n s_{n,k}^{(-1)} b(k) = \sum_{k=1}^n C_{n,k} a_k
\] 
The right-hand-side of (i) is equivalent to the expansion of the last row in the 
matrix-vector product 
\[
(C_{i,j})_{n \times n} \begin{bmatrix} a_1 \\ a_2 \\ \vdots \\ a_n \end{bmatrix} = 
     (T_{\Div})_n^{-1} \left([q^i] \frac{q^j}{1-q^j} \frac{1}{(q; q)_{\infty}}\right)_{n \times n} 
     \begin{bmatrix} a_1 \\ a_2 \\ \vdots \\ a_n \end{bmatrix}, 
\] 
and that the left-hand-side of (i) is expanded by 
\[
(s_{n,k}^{(-1)})_{n \times n} (T_{\Div})_n \begin{bmatrix} a_1 \\ a_2 \\ \vdots \\ a_n \end{bmatrix} = 
     (T_{\Div})_n^{-1} \left(p(i-j)\right)_{n \times n} (T_{\Div})_n 
     \begin{bmatrix} a_1 \\ a_2 \\ \vdots \\ a_n \end{bmatrix}. 
\] 
Then we have that the sequence $b(n)$ is expanded by multiplying the left-hand-side of (i) by the 
matrix $(p(i-j))_{n \times n}^{-1} (T_{\Div})_n$ where 
\begin{align*} 
     (p(i-j))_{n \times n}^{-1} \left([q^i] \frac{q^j}{1-q^j} \frac{1}{(q; q)_{\infty}}\right)_{n \times n} 
     \begin{bmatrix} a_1 \\ a_2 \\ \vdots \\ a_n \end{bmatrix} & = 
     \left([q^n] \frac{q^k}{1-q^k}\right) 
     \begin{bmatrix} a_1 \\ a_2 \\ \vdots \\ a_n \end{bmatrix} \\ 
     & \longmapsto 
     [q^n] \sum_{k=1}^n \frac{q^k}{1-q^k} \cdot a_k = b(n) 
     \qedhere 
\end{align*} 
\end{proof} 

\begin{cor}[An Exact Formula for a Prime Counting Function] 
For all $n \geq 1$, we have the exact formula for the function $\omega(n)$ which 
counts the number of distinct primes dividing $n$ given by \cite[A001221]{OEIS} 
\begin{align*} 
\omega(n) & = \log_2\left[\sum_{k=1}^n \sum_{j=1}^k \left(\sum_{d|k} \sum_{i=1}^d 
     p(d-ji)\right) s_{n,k} \cdot |\mu(j)|\right], 
\end{align*} 
where $s_{n,k} = s_o(n, k) - s_e(n, k)$ denotes the difference of the number of $k$'s in all 
partitions of $n$ into an odd (even) number of distinct parts as in 
Section \ref{Section_Intro}. 
\end{cor} 
\begin{proof} 
We select the special case of $(a, b) := (|\mu|, 2^{\omega})$ to arrive at the 
statement in the corollary. 
\end{proof} 

\subsection*{Acknowledgments} 

The authors thank the referees for their helpful insights and comments on 
preparing the manuscript. The author also thanks Mircea Merca for suggestions on the 
article and for collaboration on the prior articles on Lambert series factorization 
theorems.

\end{document}